\documentclass[12pt, reqno]{amsart}
\usepackage{amsmath, amsthm, amscd, amsfonts, amssymb, graphicx, color}
\usepackage[bookmarksnumbered, colorlinks, plainpages]{hyperref}

\newtheorem{theorem}{Theorem}[section]
\newtheorem{lemma}[theorem]{Lemma}
\newtheorem{proposition}[theorem]{Proposition}
\newtheorem{corollary}[theorem]{Corollary}
\theoremstyle{definition}

\theoremstyle{remark}
\newtheorem{remark}[theorem]{Remark}
\numberwithin{equation}{section}
\begin{document}
\setcounter{page}{1}

\title[$n$-dual spaces associated to a normed space]{$n$-dual spaces
associated to a normed space}
\author{Yosafat E. P. Pangalela}
\address{Department of Mathematics and Statistics, University of Otago, PO
Box 56, Dunedin 9054, New Zealand}
\email{yosafat.pangalela@maths.otago.ac.nz}
\date{Received: August 12, 2015.}

\begin{abstract}
For a real normed space $X$, we study the $n$-dual space of $\left(
X,\left\Vert \cdot \right\Vert \right) $ and show that the space is a Banach
space. Meanwhile, for a real normed space $X$ of dimension $d\geq n$ which
satisfies property (G), we discuss the $n$-dual space of $\left(
X,\left\Vert \cdot ,\ldots ,\cdot \right\Vert _{G}\right) $, where $%
\left\Vert \cdot ,\ldots ,\cdot \right\Vert _{G}$ is the G\"{a}hler $n$%
-norm. We then investigate the relationship between the $n$-dual space of $%
\left( X,\left\Vert \cdot \right\Vert \right) $ and the $n$-dual space of $%
\left( X,\left\Vert \cdot ,\ldots ,\cdot \right\Vert _{G}\right) $. We use
this relationship to determine the $n$-dual space of $\left( X,\left\Vert
\cdot ,\ldots ,\cdot \right\Vert _{G}\right) ~$and show that the space is
also a Banach space.
\end{abstract}

\maketitle

\section{Introduction}

In the 1960's, the notion of $n$-normed spaces was introduced by G\"{a}hler
\cite{G64, Ga1, Ga2, Ga3} as a generalisation of normed spaces. For every
real normed space $X$ of dimension $d\geq n$, G\"{a}hler showed that $X$ can
be viewed as an $n$-normed space by using the G\"{a}hler $n$-norm, which is
denoted by $\left\Vert \cdot ,\ldots ,\cdot \right\Vert _{G}$. This $n$-norm
is defined by using the set of bounded linear functionals on $X$. Since
then, many researchers have studied operators and functionals on $n$-normed
space $X$ (see \cite{BGP13,GGN10, LD82, L04, MC12, PG13, S12, W69}).

In \cite{PG13}, the author and Gunawan introduced the concept of $n$-dual
spaces. For every real normed space $X$ of dimension $d \geq n$, there are
two $n$-dual spaces associated to $X$. The first is the $n$-dual space of $%
\left( X,\left\Vert \cdot \right\Vert \right) $, and the other is the $n$%
-dual space of $\left( X,\left\Vert \cdot ,\ldots ,\cdot \right\Vert
_{G}\right) $. In case $X$ is the $l^{p}$ space for some $1\leq p<\infty $,
the author and Gunawan have investigated and given the relationship between
both $n$-dual spaces \cite{PG13}. Here we provide an analogues result on
more general normed spaces.

For a real normed space $X$, we investigate the $n$-dual space of $\left(
X,\left\Vert \cdot \right\Vert \right) $ by using the $\left( n-1\right) $%
-dual space of $\left( X,\left\Vert \cdot \right\Vert \right) $ (Theorem \ref%
{n-dual-space-norm}). We then focus on a real normed space $X$ of dimension $%
d\geq n$ which satisfies property (G) and discuss the relationship between
the $n$-dual space of $\left( X,\left\Vert \cdot \right\Vert \right) \,$and
the $n$-dual space of $\left( X,\left\Vert \cdot ,\ldots ,\cdot \right\Vert
_{G}\right) $ (Theorem \ref{n-dual-space-n-norm}). It is interesting to
observe that both the $n$-dual space of $\left( X,\left\Vert \cdot
\right\Vert \right) $ and the $n$-dual space of $\left( X,\left\Vert \cdot
,\ldots ,\cdot \right\Vert _{G}\right) $ are Banach spaces (Theorem \ref%
{banach-space-n-dual-space} and Theorem \ref%
{banach-space-n-dual-space-Gahler}).

\section{Preliminaries}

\label{preliminaries}

Let $n$ be a nonnegative integer and $X$ a real vector space of dimension $%
d\geq n$. We call a real-valued function $\left\Vert \cdot ,\ldots ,\cdot
\right\Vert $ on $X^{n}$ an $n$\emph{-norm }on $X$ if for all $x_{1},\ldots
,x_{n},x^{\prime }\in X$, we have

\begin{enumerate}
\item[(1)] $\left\Vert x_{1},\ldots ,x_{n}\right\Vert =0$ if and only if $%
x_{1},\ldots ,x_{n}$ are linearly dependent;

\item[(2)] $\left\Vert x_{1},\ldots ,x_{n}\right\Vert $ is invariant under
permutation;

\item[(3)] $\left\Vert \alpha x_{1},x_{2},\ldots ,x_{n}\right\Vert
=\left\vert \alpha \right\vert \left\Vert x_{1},x_{2},\ldots
,x_{n}\right\Vert $ for all $\alpha \in \mathbb{R}$; and

\item[(4)] $\left\Vert x_{1}+x^{\prime },\ldots ,x_{n}\right\Vert \leq
\left\Vert x_{1},\ldots ,x_{n}\right\Vert +\left\Vert x^{\prime },\ldots
,x_{n}\right\Vert $.
\end{enumerate}

We then call the pair $\left( X,\left\Vert \cdot ,\ldots ,\cdot \right\Vert
\right) $ an $n$\emph{-normed space}$.$

An example of $n$-normed spaces is the $l^{p}$ space, where $1\leq p<\infty $%
, equipped with
\begin{equation*}
\left\Vert x_{1},\ldots ,x_{n}\right\Vert _{p}:=\big(\frac{1}{n!}%
\sum_{j_{1}}\cdots \sum_{j_{n}}\left\vert \det \left( x_{ij_{k}}\right)
\right\vert ^{p}\big)^{\frac{1}{p}}
\end{equation*}%
for $x_{1},\ldots ,x_{n}\in l^{p}$ (see \cite[Section 3]{G01}).

Another interesting example of $n$-normed spaces is the G\"{a}hler $n$-norm
which was introduced in \cite{Ga1, Ga2, Ga3}. Let $X$ be a real normed space
of dimension $d \geq n$, and $X^{\left( 1\right) }$ the dual space of $X$. G%
\"{a}hler showed that the function $\left\Vert \cdot ,\ldots ,\cdot
\right\Vert _{G}$ which is given by%
\begin{equation*}
\left\Vert x_{1},\ldots ,x_{n}\right\Vert _{G}:=\sup_{\substack{ f_{i}\in
X^{\left( 1\right) },\left\Vert f_{i}\right\Vert \leq 1  \\ 1\leq i\leq n}}%
\left\vert \det \left[ f_{j}\left( x_{i}\right) \right] _{i,j}\right\vert
\end{equation*}%
for all $x_{1},\ldots ,x_{n}\in X$, is an $n$-norm on $X$. Hence every real
normed space $X$ can be viewed as an $n$-normed space $\left( X,\left\Vert
\cdot ,\ldots ,\cdot \right\Vert _{G}\right) $.

Let $X$ be a real normed space of dimension $d\geq n$. Any real-valued
function $f$ on $X^{n}$ is called an $n$\emph{-functional} on $X$. An $n$%
-functional $f$ is \emph{multilinear} if it satisfies two following
properties:

\begin{enumerate}
\item[(1)] $f\left( x_{1}+y_{1},\ldots ,x_{n}+y_{n}\right) =\sum_{z_{i}\in
\left\{ x_{i},y_{i}\right\} ,1\leq i\leq n}f\left( z_{1},\ldots
,z_{n}\right) $ and

\item[(2)] $f\left( \alpha _{1}x_{1},\ldots ,\alpha _{n}x_{n}\right) =\alpha
_{1}\cdots \alpha _{n-1}f\left( x_{1},\ldots ,x_{n}\right) $
\end{enumerate}

for all $x_{1},\ldots ,x_{n},y_{1},\ldots ,y_{n}\in X$ and $\alpha
_{1},\ldots ,\alpha _{n}\in \mathbb{R}$.

For multilinear $n$-functionals $f,h$ on $X$, we define an $n$-functional $%
f+h$ by
\begin{equation*}
\left( f+h\right) \left( x_{1},\ldots ,x_{n}\right) :=f\left( x_{1},\ldots
,x_{n}\right) +h\left( x_{1},\ldots ,x_{n}\right)
\end{equation*}%
for $x_{1},\ldots ,x_{n}\in X$. Then $f+h$ is also multilinear. On the other
hand, we say $f=h$ if%
\begin{equation*}
f\left( x_{1},\ldots ,x_{n}\right) =h\left( x_{1},\ldots ,x_{n}\right)
\end{equation*}%
for $x_{1},\ldots ,x_{n}\in X$.

We call an $n$-functional $f$ \emph{bounded }on a real normed space $\left(
X,\left\Vert \cdot \right\Vert \right) $ [respectively an $n$-normed space $%
\left( X,\left\Vert \cdot ,\ldots ,\cdot \right\Vert \right) $] if there
exists a constant $K>0$ such that%
\begin{equation*}
\left\vert f\left( x_{1},\ldots ,x_{n}\right) \right\vert \leq K\left\Vert
x_{1}\right\Vert \cdots \left\Vert x_{n}\right\Vert \text{ [resp. }%
\left\vert f\left( x_{1},\ldots ,x_{n}\right) \right\vert \leq K\left\Vert
x_{1},\ldots ,x_{n}\right\Vert \text{]}
\end{equation*}%
for all $x_{1},\ldots ,x_{n}\in X$.

Let $S_{n}$ denotes the group of permutations of $\left( 1,\ldots ,n\right) $%
. Recall from \cite{PG13} that every bounded multilinear $n$-functional $f$
on $\left( X,\left\Vert \cdot ,\ldots ,\cdot \right\Vert \right) $ is \emph{%
antisymmetric} in the sense that%
\begin{equation*}
f\left( x_{1},\ldots ,x_{n}\right) =\operatorname{sgn}\left( \sigma \right) f\left(
x_{\sigma \left( 1\right) },\ldots ,x_{\sigma \left( n\right) }\right)
\end{equation*}%
for $x_{1},\ldots ,x_{n}\in X$ and $\sigma \in S_{n}$. Here $\operatorname{sgn}%
\left( \sigma \right) =1$ if $\sigma $ is an even permutation, and $\operatorname{sgn%
}\left( \sigma \right) =-1$ if $\sigma $ is an odd permutation. Note that if
$f$ is antisymmetric then for any linearly dependent $x_{1},\ldots ,x_{n}\in
X$, we have $f\left( x_{1},\ldots ,x_{n}\right) =0$.

\begin{remark}
In general, we do not have the antisymmetric property for bounded
multilinear $n$-functionals on $\left( X,\left\Vert \cdot \right\Vert
\right) $.
\end{remark}

The space of bounded multilinear $n$-functionals on $\left( X,\left\Vert
\cdot \right\Vert \right) $ is called the $n$\emph{-dual space} of $\left(
X,\left\Vert \cdot \right\Vert \right) $ and denoted by $X^{\left( n\right)
} $. For $n=0$, we define $X^{\left( 0\right) }$ as $\mathbb{R}$. \ The
function $\left\Vert \cdot \right\Vert _{n,1}$ on $X^{\left( n\right) }$
where
\begin{equation*}
\left\Vert f\right\Vert _{n,1}:=\sup_{x_{1},\ldots ,x_{n}\neq 0}\frac{%
\left\vert f\left( x_{1},\ldots ,x_{n}\right) \right\vert }{\left\Vert
x_{1}\right\Vert \cdots \left\Vert x_{n}\right\Vert }
\end{equation*}%
for $f\in X^{\left( n\right) }$, defines a norm on $X^{\left( n\right) }$
and then $X^{\left( n\right) }$ is a normed space.

Meanwhile, the $n$\emph{-dual space} of $\left( X,\left\Vert \cdot ,\ldots
,\cdot \right\Vert \right) $ is the space of bounded multilinear $n$%
-functional on $\left( X,\left\Vert \cdot ,\ldots ,\cdot \right\Vert \right)
$. This space is also a normed space with the following norm%
\begin{equation*}
\left\Vert f\right\Vert _{n,n}:=\sup_{\left\Vert x_{1},\ldots
,x_{n}\right\Vert \neq 0}\frac{\left\vert f\left( x_{1},\ldots ,x_{n}\right)
\right\vert }{\left\Vert x_{1},\ldots ,x_{n}\right\Vert }\text{.}
\end{equation*}

Now let $X,Y$ be real normed spaces. We write $B\left( X,Y\right) $ to
denote the set of bounded linear operators from $X$ into $Y$. The function $%
\left\Vert \cdot \right\Vert _{\operatorname{op}}$ where%
\begin{equation*}
\left\Vert u\right\Vert _{\operatorname{op}}:=\sup_{x\neq 0}\frac{\left\Vert u\left(
x\right) \right\Vert }{\left\Vert x\right\Vert }
\end{equation*}%
for every $u\in $ $B\left( X,Y\right) \,$, is a norm on $B\left( X,Y\right) $%
. For simplification, we write $B\left( X,Y\right) $ to denote the normed
space $B\left( X,Y\right) $ equipped with $\left\Vert \cdot \right\Vert _{%
\operatorname{op}}$. Otherwise, if $\left\Vert \cdot \right\Vert ^{\ast }$ is a norm
function on $B\left( X,Y\right) $, we write $\left( B\left( X,Y\right)
,\left\Vert \cdot \right\Vert ^{\ast }\right) $ to denote the normed space $%
B\left( X,Y\right) $ equipped by the norm $\left\Vert \cdot \right\Vert
^{\ast }$.

\section{The $n$-dual space of $(X,\left\Vert \cdot \right\Vert )$}

\label{dual-space-norm}

In this section, we first identify the bounded multilinear $n$-functionals
on $(X,\left\Vert \cdot \right\Vert )$ (Proposition Proposition \ref%
{n-functional}). We then identify the $n$-dual space of $(X,\left\Vert \cdot
\right\Vert )$ by using the $\left( n-1\right) $-dual space of $%
(X,\left\Vert \cdot \right\Vert )$ (Theorem \ref{n-dual-space-norm}).
Finally we show that the $n$-dual space of $(X,\left\Vert \cdot \right\Vert
) $ is a Banach space (Theorem \ref{banach-space-n-dual-space}).

\begin{proposition}
\label{n-functional}Let $X$ be a real normed space of dimension $d\geq n$
and $f$ a bounded multilinear $n$-functional on $(X,\left\Vert \cdot
\right\Vert )$. Then there exists $u_{f}\in B(X,X^{\left( n-1\right) })$
such that for $x_{1},\ldots ,x_{n-1},z\in X$,
\begin{equation*}
f\left( x_{1},\ldots ,x_{n-1},z\right) =\left( u_{f}\left( z\right) \right)
\left( x_{1},\ldots ,x_{n-1}\right) \text{.}
\end{equation*}%
Furthermore, $\left\Vert f\right\Vert _{n,1}=\left\Vert u_{f}\right\Vert _{%
\operatorname{op}}$.
\end{proposition}

\begin{proof}
Take $z\in X$ and define an $\left( n-1\right) $-functional $f_{z}$ on $X$
with
\begin{equation*}
f_{z}\left( x_{1},\ldots ,x_{n-1}\right) :=f\left( x_{1},\ldots
,x_{n-1},z\right)
\end{equation*}%
for $x_{1},\ldots ,x_{n-1}\in X$. We show $f_{z}\in X^{\left( n-1\right) }$.
Note that for $x_{1},\ldots ,x_{n-1},y_{1},\ldots ,y_{n-1}\in X$ and $\alpha
_{1},\ldots ,\alpha _{n-1}\in \mathbb{R}$, we have
\begin{align*}
f_{z}\left( x_{1}+y_{1},\ldots ,x_{n-1}+y_{n-1}\right) & =f\left(
x_{1}+y_{1},\ldots ,x_{n-1}+y_{n-1},z\right) \\
& =\sum_{z_{i}\in \left\{ x_{i},y_{i}\right\} ,1\leq i\leq n-1}f\left(
z_{1},\ldots ,z_{n-1},z\right) \\
& =\sum_{z_{i}\in \left\{ x_{i},y_{i}\right\} ,1\leq i\leq n-1}f_{z}\left(
z_{1},\ldots ,z_{n-1}\right) \text{,}
\end{align*}%
\begin{align*}
f_{z}\left( \alpha _{1}x_{1},\ldots ,\alpha _{n-1}x_{n-1}\right) & =f\left(
\alpha _{1}x_{1},\ldots ,\alpha _{n-1}x_{n-1},z\right) \\
& =\alpha _{1}\cdots \alpha _{n-1}f\left( x_{1},\ldots ,x_{n-1},z\right) \\
& =\alpha _{1}\cdots \alpha _{n-1}f_{z}\left( x_{1},\ldots ,x_{n-1}\right)
\text{,}
\end{align*}%
and%
\begin{equation*}
\left\vert f_{z}\left( x_{1},\ldots ,x_{n-1}\right) \right\vert =\left\vert
f\left( x_{1},\ldots ,x_{n-1},z\right) \right\vert \leq \left\Vert
f\right\Vert _{n,1}\left\Vert z\right\Vert \left( \left\Vert
x_{1}\right\Vert \cdots \left\Vert x_{n-1}\right\Vert \right)
\end{equation*}%
since $f$ is bounded on $(X,\left\Vert \cdot \right\Vert )$. Hence $%
f_{z}:X^{n-1}\rightarrow \mathbb{R}$ is multilinear and bounded; and then $%
f_{z}\in X^{\left( n-1\right) }$.

Now define $u_{f}:X\rightarrow X^{\left( n-1\right) }$ with $u_{f}\left(
z\right) :=f_{z}$ for $z\in X$. We have to show $u_{f}\in B(X,X^{\left(
n-1\right) })$. First we show that $u_{f}$ is linear. Take $z_{1},z_{2}\in X$
and $\alpha ,\beta \in \mathbb{R}$. For every $x_{1},\ldots ,x_{n-1}\in X$,
we have%
\begin{align*}
\left( u_{f}\left( \alpha z_{1}+\beta z_{2}\right) \right) \left(
x_{1},\ldots ,x_{n-1}\right) & =f_{\alpha z_{1}+\beta z_{2}}\left(
x_{1},\ldots ,x_{n-1}\right) \\
& =f\left( x_{1},\ldots ,x_{n-1},\alpha z_{1}+\beta z_{2}\right) \\
& =f\left( x_{1},\ldots ,x_{n-1},\alpha z_{1}\right) +f\left( x_{1},\ldots
,x_{n-1},\beta z_{2}\right) \\
& =\alpha f\left( x_{1},\ldots ,x_{n-1},z_{1}\right) +\beta f\left(
x_{1},\ldots ,x_{n-1},z_{2}\right) \\
& =\alpha f_{z_{1}}\left( x_{1},\ldots ,x_{n-1}\right) +\beta
f_{z_{2}}\left( x_{1},\ldots ,x_{n-1}\right) \\
& =\left( \alpha u_{f}\left( z_{1}\right) \right) \left( x_{1},\ldots
,x_{n-1}\right) +\left( \beta u_{f}\left( z_{2}\right) \right) \left(
x_{1},\ldots ,x_{n-1}\right) \\
& =\left( \alpha u_{f}\left( z_{1}\right) +\beta u_{f}\left( z_{2}\right)
\right) \left( x_{1},\ldots ,x_{n-1}\right)
\end{align*}%
and%
\begin{equation*}
u_{f}\left( \alpha z_{1}+\beta z_{2}\right) =\alpha u_{f}\left( z_{1}\right)
+\beta u_{f}\left( z_{2}\right) \text{.}
\end{equation*}%
Hence $u_{f}$ is linear.

Next we show the boundedness of $u_{f}$. Take $z\in X$. Then for $%
x_{1},\ldots ,x_{n-1}\in X$, we have%
\begin{align*}
\left\vert \left( u_{f}\left( z\right) \right) \left( x_{1},\ldots
,x_{n-1}\right) \right\vert & =\left\vert f_{z}\left( x_{1},\ldots
,x_{n-1}\right) \right\vert =\left\vert f\left( x_{1},\ldots
,x_{n-1},z\right) \right\vert \\
& \leq \left\Vert f\right\Vert _{n,1}\left\Vert x_{1}\right\Vert \cdots
\left\Vert x_{n-1}\right\Vert \left\Vert z\right\Vert \text{ (}f\text{ is
bounded on }(X,\left\Vert \cdot \right\Vert )\text{)}
\end{align*}%
and then%
\begin{equation*}
\left\Vert u_{f}\left( z\right) \right\Vert =\sup_{x_{1},\ldots ,x_{n-1}\neq
0}\frac{\left\vert \left( u_{f}\left( z\right) \right) \left( x_{1},\ldots
,x_{n-1}\right) \right\vert }{\left\Vert x_{1}\right\Vert \cdots \left\Vert
x_{n-1}\right\Vert }\leq \left\Vert f\right\Vert _{n,1}\left\Vert
z\right\Vert
\end{equation*}%
which is finite. This implies%
\begin{equation*}
\sup_{z\neq 0}\frac{\left\Vert u_{f}\left( z\right) \right\Vert }{\left\Vert
z\right\Vert }\leq \left\Vert f\right\Vert _{n,1}
\end{equation*}%
which is finite. Therefore $u_{f}$ is bounded and $\left\Vert
u_{f}\right\Vert _{\operatorname{op}}\leq \left\Vert f\right\Vert _{n,1}$.

Finally we claim that $\left\Vert u_{f}\right\Vert _{\operatorname{op}}=\left\Vert
f\right\Vert _{n,1}$. Recall that we already have $\left\Vert
u_{f}\right\Vert _{\operatorname{op}}\leq \left\Vert f\right\Vert _{n,1}$. To show
the reverse inequality, note that for $z\in X$, $u_{f}\left( z\right) =f_{z}$
is bounded. Then for $x_{1},\ldots ,x_{n-1},z\in X$,%
\begin{align*}
\left\vert f\left( x_{1},\ldots ,x_{n-1},z\right) \right\vert & =\left\vert
f_{z}\left( x_{1},\ldots ,x_{n-1}\right) \right\vert \\
& \leq \left\Vert f_{z}\right\Vert _{n,1}\left\Vert x_{1}\right\Vert \cdots
\left\Vert x_{n-1}\right\Vert \text{ } \\
& \text{ \ \ \ \ \ (}f\text{ is bounded on }(X,\left\Vert \cdot \right\Vert )%
\text{)} \\
& =\left\Vert u_{f}\left( z\right) \right\Vert \left\Vert x_{1}\right\Vert
\cdots \left\Vert x_{n-1}\right\Vert \\
& \leq (\left\Vert u_{f}\right\Vert _{\operatorname{op}}\left\Vert z\right\Vert
)\left\Vert x_{1}\right\Vert \cdots \left\Vert x_{n-1}\right\Vert
\end{align*}%
since $u_{f}$ is bounded. Hence%
\begin{equation*}
\left\Vert f\right\Vert _{n,1}=\sup_{x_{1},\ldots ,x_{n-1},z\neq 0}\frac{%
\left\vert f\left( x_{1},\ldots ,x_{n-1},z\right) \right\vert }{\left\Vert
x_{1}\right\Vert \cdots \left\Vert x_{n-1}\right\Vert \left\Vert
z\right\Vert }\leq \left\Vert u_{f}\right\Vert _{\operatorname{op}}
\end{equation*}%
and $\left\Vert f\right\Vert _{n,1}\leq \left\Vert u_{f}\right\Vert _{\operatorname{%
op}}$. Therefore $\left\Vert u_{f}\right\Vert _{\operatorname{op}}=\left\Vert
f\right\Vert _{n,1}$, as claimed.
\end{proof}

\begin{theorem}
\label{n-dual-space-norm}Let $X$ be a real normed space of dimension $d\geq
n $. Then the $n$-dual space of $(X,\left\Vert \cdot \right\Vert )$ is $%
B(X,X^{\left( n-1\right) })$.
\end{theorem}

\begin{proof}
For a bounded multilinear $n$-functional on $(X,\left\Vert \cdot \right\Vert
)$ $f$, let $u_{f}\in B(X,X^{\left( n-1\right) })$ be as in Proposition \ref%
{n-functional}. Define a map $\theta $ from the $n$-dual space of $%
(X,\left\Vert \cdot \right\Vert )$ to $B\left( X,X^{\left( n-1\right)
}\right) $ with%
\begin{equation*}
\theta \left( f\right) :=u_{f}
\end{equation*}%
for $f\in X^{\left( n\right) }$. We have to show that $\theta $ is isometric
and bijective.

The isometricness of $\theta $ follows from Proposition \ref{n-functional}.

Next we show the injectivity of $\theta $. Let $f,h$ be bounded multilinear $%
n$-functionals on $(X,\left\Vert \cdot \right\Vert )$ such that $\theta
\left( f\right) =\theta \left( h\right) $. Then $u_{f}=u_{h}$ and for every $%
x_{1},\ldots ,x_{n-1},x_{n}\in X$, we have
\begin{align*}
f\left( x_{1},\ldots ,x_{n-1},x_{n}\right) & =\left( u_{f}\left(
x_{n}\right) \right) \left( x_{1},\ldots ,x_{n-1}\right) \\
& =\left( u_{h}\left( x_{n}\right) \right) \left( x_{1},\ldots
,x_{n-1}\right) \\
& =h\left( x_{1},\ldots ,x_{n-1},x_{n}\right) \text{.}
\end{align*}%
Hence $f=h$ and $\theta $ is injective.

To show that $\theta $ is surjective, we take $u\in B\left( X,X^{\left(
n-1\right) }\right) $ and have to show that there exists a bounded
multilinear $n$-functional on $(X,\left\Vert \cdot \right\Vert )$ $f_{u}$
such that $\theta \left( f_{u}\right) =u$. Now we define $f_{u}$ an $n$%
-functional on $X$ where%
\begin{equation*}
f_{u}\left( x_{1},\ldots ,x_{n-1},x_{n}\right) :=\left( u\left( x_{n}\right)
\right) \left( x_{1},\ldots ,x_{n-1}\right)
\end{equation*}%
for $x_{1},\ldots ,x_{n-1},x_{n}\in X$. We claim that $f_{u}$ is multilinear
and bounded on $(X,\left\Vert \cdot \right\Vert )$.

First we show that $f_{u}$ is multilinear. Take $x_{1},\ldots
,x_{n},y_{1},\ldots ,y_{n}\in X$ and $\alpha _{1},\ldots ,\alpha _{n}\in
\mathbb{R}$. we have
\begin{align*}
f_{u}\left( x_{1}+y_{1},\ldots ,x_{n}+y_{n}\right) & =\left( u\left(
x_{n}+y_{n}\right) \right) \left( x_{1}+y_{1},\ldots ,x_{n-1}+y_{n-1}\right)
\\
& =\sum_{z_{i}\in \left\{ x_{i},y_{i}\right\} ,1\leq i\leq n-1}\left(
u\left( x_{n}+y_{n}\right) \right) \left( z_{1},\ldots ,z_{n-1}\right) \\
& =\sum_{z_{i}\in \left\{ x_{i},y_{i}\right\} ,1\leq i\leq n-1}\left(
u\left( x_{n}\right) +u\left( y_{n}\right) \right) \left( z_{1},\ldots
,z_{n-1}\right) \\
& =\sum_{z_{i}\in \left\{ x_{i},y_{i}\right\} ,1\leq i\leq n-1}\left(
f\left( z_{1},\ldots ,z_{n-1},x_{n}\right) +f\left( z_{1},\ldots
,z_{n-1},y_{n}\right) \right) \\
& =\sum_{z_{i}\in \left\{ x_{i},y_{i}\right\} ,1\leq i\leq n}f\left(
z_{1},\ldots ,z_{n-1},z_{n}\right)
\end{align*}%
and%
\begin{align*}
f_{u}\left( \alpha _{1}x_{1},\ldots ,\alpha _{n}x_{n}\right) & =\left(
u\left( \alpha _{n}x_{n}\right) \right) \left( \alpha _{1}x_{1},\ldots
,\alpha _{n-1}x_{n-1}\right) \\
& =\alpha _{1}\cdots \alpha _{n-1}\left( u\left( \alpha _{n}x_{n}\right)
\right) \left( x_{1},\ldots ,x_{n-1}\right) \text{ (}u\left( \alpha
_{n}x_{n}\right) \text{ is multilinear)} \\
& =\alpha _{1}\cdots \alpha _{n-1}\alpha _{n}\left( u\left( x_{n}\right)
\right) \left( x_{1},\ldots ,x_{n-1}\right) \text{ (}u\text{ is linear)} \\
& =\alpha _{1}\cdots \alpha _{n-1}\alpha _{n}f_{u}\left( x_{1},\ldots
,x_{n-1},x_{n}\right) \text{.}
\end{align*}%
Hence $f_{u}$ is multilinear.

Next we show that $f_{u}$ is bounded on $(X,\left\Vert \cdot \right\Vert )$.
Take $x_{1},\ldots ,x_{n}\in X$. Then%
\begin{align*}
\left\vert f_{u}\left( x_{1},\ldots ,x_{n-1},x_{n}\right) \right\vert &
=\left\vert \left( u\left( x_{n}\right) \right) \left( x_{1},\ldots
,x_{n-1}\right) \right\vert \\
& \leq \left\Vert u\left( x_{n}\right) \right\Vert \left\Vert
x_{1}\right\Vert \cdots \left\Vert x_{n-1}\right\Vert \text{ (}u\left(
x_{n}\right) \text{ is bounded)} \\
& \leq (\left\Vert u\right\Vert _{\operatorname{op}}\left\Vert x_{n}\right\Vert
)\left\Vert x_{1}\right\Vert \cdots \left\Vert x_{n-1}\right\Vert \text{ (}u%
\text{ is bounded)}
\end{align*}%
and $f_{u}$ is bounded.

Hence $f_{u}$ is multilinear and bounded on $(X,\left\Vert \cdot \right\Vert
)$, as claimed. Note that $\theta \left( f_{u}\right) =u_{f_{u}}$. Take $%
x_{1},\ldots ,x_{n}\in X$ and we have%
\begin{equation*}
\left( u\left( x_{n}\right) \right) \left( x_{1},\ldots ,x_{n-1}\right)
=f_{u}\left( x_{1},\ldots ,x_{n-1},x_{n}\right) =\left( \left(
u_{f_{u}}\right) \left( x_{n}\right) \right) \left( x_{1},\ldots
,x_{n-1}\right) \text{.}
\end{equation*}%
Then $u\left( x_{n}\right) =u_{f_{u}}\left( x_{n}\right) $ for $x_{n}\in X$,
and
\begin{equation*}
u=u_{f_{u}}=\theta \left( f_{u}\right) \text{.}
\end{equation*}
Therefore, $\theta $ is surjective and a bijection, as required.
\end{proof}

Recall from \cite[Theorem 2.10-2]{KRE} that for normed spaces $X,Y$, the
normed space $B\left( X,Y\right) $ is a Banach space if $Y$ is a Banach
space. Since $\mathbb{R}$ is a Banach space, then for every normed space $X$%
, $X^{\left( 1\right) }$ is also a Banach space. Hence Theorem \ref%
{n-dual-space-norm} with $n=2$ implies that $X^{\left( 2\right) }$ is also a
Banach space. Therefore, by induction and Theorem \ref{n-dual-space-norm},
we get the following theorem.

\begin{theorem}
\label{banach-space-n-dual-space}Let $X$ be a real normed space of dimension
$d\geq n$. Then the $n$-dual space of $(X,\left\Vert \cdot \right\Vert )$ is
a Banach space.
\end{theorem}

\section{The $n$-dual space of $(X,\left\Vert \cdot ,\cdots ,\cdot
\right\Vert _{G})$}

\label{dual-space-n-norm}

In this section, we focus on normed spaces of dimension $d\geq n$ which
satisfy property (G). On this space, we investigate the relationship between
bounded multilinear $n$-functionals on $(X,\left\Vert \cdot ,\cdots ,\cdot
\right\Vert _{G})$ and bounded multilinear $n$-functionals on $(X,\left\Vert
\cdot \right\Vert )$ (Lemma \ref{equivalence}). We then use it to determine
the $n$-dual space of $(X,\left\Vert \cdot ,\cdots ,\cdot \right\Vert _{G})$
(Theorem \ref{n-dual-space-n-norm}) and show that the space is a Banach
space (Theorem \ref{banach-space-n-dual-space-Gahler}).

First we recall the functional $g$ and property (G) introduced by Mili\v{c}i%
\'{c} in \cite{M93}. The functional $g:X^{2}\rightarrow \mathbb{R}$ is
defined by%
\begin{equation*}
g\left( x,y\right) :=\frac{\left\Vert x\right\Vert }{2}\left( \tau
_{-}\left( x,y\right) +\tau _{+}\left( x,y\right) \right)
\end{equation*}%
where%
\begin{equation*}
\tau _{\pm }\left( x,y\right) :=\lim_{t\rightarrow \pm 0}t^{-1}\left(
\left\Vert x+ty\right\Vert -\left\Vert x\right\Vert \right) \text{.}
\end{equation*}%
The functional $g$ satisfies the following properties: for all $x,y\in X$
and $\alpha ,\beta \in \mathbb{R}$

\begin{enumerate}
\item[(G1)] $g\left( x,x\right) =\left\Vert x\right\Vert ^{2}$;

\item[(G2)] $g\left( \alpha x,\beta y\right) =\alpha \beta g\left(
x,y\right) $;

\item[(G3)] $g\left( x,x+y\right) =\left\Vert x\right\Vert ^{2}+g\left(
x,y\right) $; and

\item[(G4)] $\left\vert g\left( x,y\right) \right\vert \leq \left\Vert
x\right\Vert \left\Vert y\right\Vert $.
\end{enumerate}

We say that a real normed space $X$ satisfies \emph{property (G)} if the
functional $g\left( x,y\right) $ is linear with respect to $y\in X$. In that
case, we then call $g$ a \emph{semi-inner product }on $X$. For example, for $%
1\leq p<\infty $, the $l^{p}$ space satisifes property (G)\ (see \cite{WG13}%
).

By using the semi-inner product $g$, we define an orthogonal relation on $X$
as follows:%
\begin{equation*}
x\perp _{g}y\Leftrightarrow g\left( x,y\right) =0\text{.}
\end{equation*}

Let $x\in X$ and $Y=\left\{ y_{1},\ldots ,y_{n}\right\} \subseteq X$. We
write $\Gamma \left( y_{1},\ldots ,y_{n}\right) $ to denote the Gram
determinant $\det \left[ g\left( y_{i},y_{j}\right) \right] _{i,j}$. If $%
\Gamma \left( y_{1},\ldots ,y_{n}\right) \neq 0$, then the vector%
\begin{equation*}
x_{Y}:=-\frac{1}{\Gamma \left( y_{1},\ldots ,y_{n}\right) }\det \left[
\begin{array}{cccc}
0 & y_{1} & \cdots & y_{n} \\
g\left( y_{1},x\right) & g\left( y_{1},y_{1}\right) & \cdots & g\left(
y_{1},y_{n}\right) \\
\vdots & \vdots &  & \vdots \\
g\left( y_{n},x\right) & g\left( y_{n},y_{1}\right) & \cdots & g\left(
y_{n},y_{n}\right)%
\end{array}%
\right]
\end{equation*}%
is called the Gram-Schimdt projection of the vector $x$ on $Y$.

Next let $\left\{ x_{1},\ldots ,x_{n}\right\} $ be a linearly independent
set of vectors in $X$. As in \cite{M93}, we call $x_{1}^{\circ },\ldots
,x_{n}^{\circ }$ the \emph{left }$g$\emph{-orthogonal sequence} where $%
x_{1}^{\circ }:=x_{1}$ and for $i=2,\ldots ,n$,%
\begin{equation*}
x_{i}^{\circ }:=x_{i}-\left( x_{i}\right) _{S_{i-1}}\text{,}
\end{equation*}%
where $S_{i-1}:=\operatorname{span}\left\{ x_{1},\ldots ,x_{i-1}\right\} $. Note
that if $i<j$, then $x_{i}^{\circ }\perp _{g}x_{j}^{\circ }$ and $%
g(x_{i}^{\circ },x_{j}^{\circ })=0$.

\begin{proposition}
\label{n-norm-inequality}Let $X$ be a real normed space of dimension $d\geq
n $ which satisfies property (G). Let $\left\{ x_{1},\ldots ,x_{n}\right\} $
be a linearly independent set of vectors in $X$. Then%
\begin{equation*}
\left\Vert x_{1}^{\circ }\right\Vert \cdots \left\Vert x_{n}^{\circ
}\right\Vert \leq \left\Vert x_{1},\ldots ,x_{n}\right\Vert _{G}\leq
n!\left\Vert x_{1}\right\Vert \cdots \left\Vert x_{n}\right\Vert \text{.}
\end{equation*}
\end{proposition}

\begin{proof}
First we show the right inequality. Note that%
\begin{align*}
\left\Vert x_{1},\ldots ,x_{n}\right\Vert _{G}& =\sup_{\substack{ f_{i}\in
X^{\left( 1\right) },\left\Vert f_{i}\right\Vert \leq 1  \\ 1\leq i\leq n}}%
\left\vert \det \left[ f_{j}\left( x_{i}\right) \right] _{i,j}\right\vert \\
& =\sup_{\substack{ f_{i}\in X^{\left( 1\right) },\left\Vert
f_{i}\right\Vert \leq 1  \\ 1\leq i\leq n}}\left\vert \sum_{\sigma \in S_{n}}%
\operatorname{sgn}\left( \sigma \right) \prod_{i=1}^{n}f_{\sigma \left( i\right)
}\left( x_{i}\right) \right\vert \text{ (by the Leibniz formula)} \\
& \leq \sup_{\substack{ f_{i}\in X^{\left( 1\right) },\left\Vert
f_{i}\right\Vert \leq 1  \\ 1\leq i\leq n}}\sum_{\sigma \in S_{n}}\left\vert
\prod_{i=1}^{n}f_{\sigma \left( i\right) }\left( x_{i}\right) \right\vert
\text{ (by the triangle inequality)} \\
& \leq \sup_{\substack{ f_{i}\in X^{\left( 1\right) },\left\Vert
f_{i}\right\Vert \leq 1  \\ 1\leq i\leq n}}\sum_{\sigma \in S_{n}}\Big(%
\prod_{i=1}^{n}\left\Vert f_{\sigma \left( i\right) }\right\Vert \left\Vert
x_{i}\right\Vert \Big)\text{ (each }f_{i}\text{ is bounded)} \\
& \leq \sum_{\sigma \in S_{n}}\Big(\prod_{i=1}^{n}\left\Vert
x_{i}\right\Vert \Big) \\
& =n!\left\Vert x_{1}\right\Vert \cdots \left\Vert x_{n}\right\Vert \text{,}
\end{align*}%
as required.

To show the left inequality, we first show that for a fixed $x\in X$, the
functional $g_{x}$ on $X$ defined by
\begin{equation*}
g_{x}\left( y\right) :=\frac{g\left( x,y\right) }{\left\Vert x\right\Vert }
\end{equation*}%
for $y\in X$, is bounded and linear. The linearity follows since $X$
satisfies property (G). Now take $y\in X$, by (G4), we have%
\begin{equation*}
\left\vert g_{x}\left( y\right) \right\vert =\left\vert \frac{g\left(
x,y\right) }{\left\Vert x\right\Vert }\right\vert \leq \left\Vert
y\right\Vert
\end{equation*}%
and $g_{x}$ is bounded, as required. Hence for $x\in X$, $g_{x}\in X^{\left(
1\right) }$. Furthermore, $\left\Vert g_{x}\right\Vert \leq 1$.

Now note that $\left\Vert x_{1},\ldots ,x_{n}\right\Vert _{G}=\left\Vert
x_{1}^{\circ },\ldots ,x_{n}^{\circ }\right\Vert _{G}$. This implies%
\begin{align}
\left\Vert x_{1},\ldots ,x_{n}\right\Vert _{G}& =\left\Vert x_{1}^{\circ
},\ldots ,x_{n}^{\circ }\right\Vert _{G}=\sup_{\substack{ f_{i}\in X^{\left(
1\right) },\left\Vert f_{i}\right\Vert \leq 1  \\ 1\leq i\leq n}}\left\vert
\det \left[ f_{j}\left( x_{i}^{\circ }\right) \right] _{i,j}\right\vert
\label{ineq} \\
& \geq \left\vert \det [g_{x_{j}^{\circ }}(x_{i}^{\circ })]_{i,j}\right\vert
=\frac{1}{\left\Vert x_{1}^{\circ }\right\Vert \cdots \left\Vert
x_{n}^{\circ }\right\Vert }\left\vert \det \left[ g(x_{j}^{\circ
},x_{i}^{\circ })\right] _{i,j}\right\vert \text{.}  \notag
\end{align}%
Since $x_{1}^{\circ },\ldots ,x_{n}^{\circ }$ is the left $g$-orhogonal
sequence, then $g(x_{i}^{\circ },x_{j}^{\circ })=0$ if $i<j$. By (G1), we
get $g(x_{i}^{\circ },x_{i}^{\circ })=\left\Vert x_{i}^{\circ }\right\Vert
^{2}$ for $i=1,\ldots ,n$. This implies
\begin{equation*}
\left\vert \det \left[ g(x_{j}^{\circ },x_{i}^{\circ })\right]
_{i,j}\right\vert =\left\Vert x_{1}^{\circ }\right\Vert ^{2}\cdots
\left\Vert x_{n}^{\circ }\right\Vert ^{2}
\end{equation*}%
and \eqref{ineq} become%
\begin{equation*}
\left\Vert x_{1},\ldots ,x_{n}\right\Vert _{G}\geq \left\Vert x_{1}^{\circ
}\right\Vert \cdots \left\Vert x_{n}^{\circ }\right\Vert \text{,}
\end{equation*}%
as required.
\end{proof}

\begin{remark}
Proposition \ref{n-norm-inequality} is a generalisation of Theorem 2.2 in
\cite{WG13}. In \cite[Theorem 2.2]{WG13}, Wibawa-Kusumah and Gunawan only
proved Proposition \ref{n-norm-inequality} for $l^{p}$ spaces where $1\leq
p<\infty $.
\end{remark}

\begin{lemma}
\label{equivalence}Let $X$ be a real normed space of dimension $d\geq n$
which satisfies property (G). Let $f$ be a multilinear $n$-functional on $X$%
. Then $f$ is antisymmetric and bounded on $(X,\left\Vert \cdot \right\Vert
) $ if and only if $f$ is bounded on $(X,\left\Vert \cdot ,\cdots ,\cdot
\right\Vert _{G})$. Furthermore%
\begin{equation*}
\left\Vert f\right\Vert _{n,n}\leq \left\Vert f\right\Vert _{n,1}\leq
n!\left\Vert f\right\Vert _{n,n}\text{.}
\end{equation*}
\end{lemma}

\begin{proof}
First suppose that $f$ is antisymmetric bounded on $(X,\left\Vert \cdot
\right\Vert )$. Take linearly independent $x_{1},\ldots ,x_{n}\in X$. Then
\begin{equation*}
f\left( x_{1},\ldots ,x_{n}\right) =f\left( x_{1}^{\circ },\ldots
,x_{n}^{\circ }\right)
\end{equation*}%
and by the left inequality in Proposition \ref{n-norm-inequality},%
\begin{eqnarray*}
\frac{\left\vert f\left( x_{1},\ldots ,x_{n}\right) \right\vert }{\left\Vert
x_{1},\ldots ,x_{n}\right\Vert _{G}} &\leq &\frac{\left\vert f\left(
x_{1},\ldots ,x_{n}\right) \right\vert }{\left\Vert x_{1}^{\circ
}\right\Vert \cdots \left\Vert x_{n}^{\circ }\right\Vert }=\frac{\left\vert
f\left( x_{1}^{\circ },\ldots ,x_{n}^{\circ }\right) \right\vert }{%
\left\Vert x_{1}^{\circ }\right\Vert \cdots \left\Vert x_{n}^{\circ
}\right\Vert } \\
&\leq &\left\Vert f\right\Vert _{n,1}\text{ (}f\text{ is bounded on }%
(X,\left\Vert \cdot \right\Vert )\text{)}
\end{eqnarray*}%
which is finite. Hence $f$ is bounded on $(X,\left\Vert \cdot ,\cdots ,\cdot
\right\Vert _{G})$ and%
\begin{equation}
\left\Vert f\right\Vert _{n,n}\leq \left\Vert f\right\Vert _{n,1}\text{.}
\label{Eq1}
\end{equation}

Next suppose that $f$ is bounded on $(X,\left\Vert \cdot ,\cdots ,\cdot
\right\Vert _{G})$. Then $f$ is antisymmetric. To show the boundedness of $f$
on $(X,\left\Vert \cdot \right\Vert )$, we take linearly independent $%
x_{1},\ldots ,x_{n}\in X$. Then by the right inequality in Proposition \ref%
{n-norm-inequality},%
\begin{eqnarray*}
\frac{\left\vert f\left( x_{1},\ldots ,x_{n}\right) \right\vert }{\left\Vert
x_{1}\right\Vert \cdots \left\Vert x_{n}\right\Vert } &\leq &n!\frac{%
\left\vert f\left( x_{1},\ldots ,x_{n}\right) \right\vert }{\left\Vert
x_{1},\ldots ,x_{n}\right\Vert _{G}} \\
&\leq &n!\left\Vert f\right\Vert _{n,n}\text{ (}f\text{ is bounded on }%
(X,\left\Vert \cdot ,\cdots ,\cdot \right\Vert _{G})\text{)}
\end{eqnarray*}%
which is finite. Hence $f$ is bounded on $(X,\left\Vert \cdot \right\Vert )$
and%
\begin{equation}
\left\Vert f\right\Vert _{n,1}\leq n!\left\Vert f\right\Vert _{n,n}\text{.}
\label{Eq2}
\end{equation}

Finally, by \eqref{Eq1} and \eqref{Eq2}, we get%
\begin{equation*}
\left\Vert f\right\Vert _{n,n}\leq \left\Vert f\right\Vert _{n,1}\leq
n!\left\Vert f\right\Vert _{n,n}\text{,}
\end{equation*}%
as required.
\end{proof}

Now we say $u\in B(X,X^{\left( n-1\right) })$ \emph{antisymmetric} if for $%
x_{1},\ldots ,x_{n}\in X$ and $\sigma \in S_{n}$,%
\begin{equation*}
\left( u\left( x_{n}\right) \right) \left( x_{1},\ldots ,x_{n-1}\right) =%
\operatorname{sgn}\left( \sigma \right) \left( u\left( x_{\sigma \left( n\right)
}\right) \right) \left( x_{\sigma \left( 1\right) },\ldots ,x_{\sigma \left(
n-1\right) }\right)
\end{equation*}%
and then define $B_{\operatorname{as}}(X,X^{\left( n-1\right) })$ as the collection
of antisymmetric elements of $B(X,X^{\left( n-1\right) })$. Note that $B_{%
\operatorname{as}}(X,X^{\left( n-1\right) })$ is also a normed space with the norm
inherited from $B(X,X^{\left( n-1\right) })$ which is $\left\Vert \cdot
\right\Vert _{\operatorname{op}}$.

Note that Theorem \ref{n-dual-space-norm} and Lemma \ref{equivalence} imply
that every bounded multilinear $n$-functional on $(X,\left\Vert \cdot
,\cdots ,\cdot \right\Vert _{G})$ can be identified as an element of $B_{%
\operatorname{as}}(X,X^{\left( n-1\right) })$ and vice versa. Therefore Lemma \ref%
{equivalence} implies the following corollary and theorem.

\begin{corollary}
\label{equivalent-norm}Let $X$ be a real normed space of dimension $d\geq n$
which satisfies property (G). The function $\left\Vert \cdot \right\Vert
_{G} $ on $B_{\operatorname{as}}(X,X^{\left( n-1\right) })$ where%
\begin{equation*}
\left\Vert u\right\Vert _{G}:=\sup_{\left\Vert x_{1},\ldots
,x_{n}\right\Vert _{G}\neq 0}\frac{\left\vert \left( u\left( x_{n}\right)
\right) \left( x_{1},\ldots ,x_{n-1}\right) \right\vert }{\left\Vert
x_{1},\ldots ,x_{n}\right\Vert _{G}}
\end{equation*}%
for $u\in B(X,X^{\left( n-1\right) })$, defines a norm on $B_{\operatorname{as}%
}(X,X^{\left( n-1\right) })$. Furthermore, $\left\Vert \cdot \right\Vert
_{G} $ and $\left\Vert \cdot \right\Vert _{\operatorname{op}}$ are equivalent norms
on $B_{\operatorname{as}}(X,X^{\left( n-1\right) })$ with%
\begin{equation*}
\left\Vert u\right\Vert _{G}\leq \left\Vert u\right\Vert _{\operatorname{op}}\leq
n!\left\Vert u\right\Vert _{G}
\end{equation*}%
for $u\in B(X,X^{\left( n-1\right) })$.
\end{corollary}

\begin{theorem}
\label{n-dual-space-n-norm}Let $X$ be a real normed space of dimension $%
d\geq n$ which satisfies property (G). Then the $n$-dual space of $%
(X,\left\Vert \cdot ,\cdots ,\cdot \right\Vert _{G})$ is $\left( B_{\operatorname{as}%
}(X,X^{\left( n-1\right) }),\left\Vert \cdot \right\Vert _{G}\right) $.
\end{theorem}

The rest of this section is devoted to show that for $n\in \mathbb{N}$, the $%
n$-dual space of $(X,\left\Vert \cdot ,\cdots ,\cdot \right\Vert _{G})$ is a
Banach space.

\begin{theorem}
\label{banach-space-Bss}Let $X$ be a real normed space of dimension $d\geq n$
which satisfies property (G). Then $B_{\operatorname{as}}(X,X^{\left( n-1\right) })$
is a Banach space.
\end{theorem}

\begin{proof}
Since every closed subspace of Banach space is also a Banach space, then by
Theorem \ref{banach-space-n-dual-space}, it suffices to show that $B_{\operatorname{%
as}}(X,X^{\left( n-1\right) })$ is a closed subspace of $B(X,X^{\left(
n-1\right) })$.

Take a sequence $\left\{ u_{m}\right\} \subseteq B_{\operatorname{as}}(X,X^{\left(
n-1\right) })$ such that $u_{m}\rightarrow u$. We have to show $u\in B_{%
\operatorname{as}}(X,X^{\left( n-1\right) })$. In other words, for $x_{1},\ldots
,x_{n}\in X$ and $\sigma \in S_{n}$, we have to show%
\begin{equation*}
\left( u\left( x_{n}\right) \right) \left( x_{1},\ldots ,x_{n-1}\right) =%
\operatorname{sgn}\left( \sigma \right) \left( u\left( x_{\sigma \left( n\right)
}\right) \right) \left( x_{\sigma \left( 1\right) },\ldots ,x_{\sigma \left(
n-1\right) }\right) \text{.}
\end{equation*}%
Take $x_{1},\ldots ,x_{n}\in X$ and $\sigma \in S_{n}$. First note that for $%
m\in \mathbb{N}$, we have
\begin{equation}
\left\Vert u\left( x_{n}\right) -u_{m}\left( x_{n}\right) \right\Vert
=\left\Vert \left( u-u_{m}\right) \left( x_{n}\right) \right\Vert \leq
\left\Vert u-u_{m}\right\Vert _{\operatorname{op}}\left\Vert x_{n}\right\Vert
\label{Eq3}
\end{equation}%
since $u-u_{m}$ is bounded. Since $u\left( x_{n}\right) ,u_{m}\left(
x_{n}\right) \in X^{\left( n-1\right) }$, then $\left( u-u_{m}\right) \left(
x_{n}\right) $ is bounded and for $y_{1},\ldots ,y_{n-1}\in X$, we have%
\begin{equation}
\left\Vert \left( \left( u-u_{m}\right) \left( x_{n}\right) \right) \left(
y_{1},\ldots ,y_{n-1}\right) \right\Vert \leq \left\Vert u\left(
x_{n}\right) -u_{m}\left( x_{n}\right) \right\Vert \left\Vert
y_{1}\right\Vert \cdots \left\Vert y_{n-1}\right\Vert \text{.}  \label{Eq4}
\end{equation}%
Since $u_{m}\rightarrow u$, then by \eqref{Eq3} and \eqref{Eq4}, we get%
\begin{equation}
\left( u_{m}\left( x_{n}\right) \right) \left( y_{1},\ldots ,y_{n-1}\right)
\rightarrow \left( u\left( x_{n}\right) \right) \left( y_{1},\ldots
,y_{n-1}\right)  \label{Eq5}
\end{equation}%
for $y_{1},\ldots ,y_{n-1}\in X$. Since $u_{m}$ is antisymmetric for every $%
m\in \mathbb{N}$, then \eqref{Eq5} implies%
\begin{equation*}
\left( u\left( x_{n}\right) \right) \left( x_{1},\ldots ,x_{n-1}\right) =%
\operatorname{sgn}\left( \sigma \right) \left( u\left( x_{\sigma \left( n\right)
}\right) \right) \left( x_{\sigma \left( 1\right) },\ldots ,x_{\sigma \left(
n-1\right) }\right) \text{,}
\end{equation*}%
as required. Thus $B_{\operatorname{as}}(X,X^{\left( n-1\right) })$ is closed and
then a Banach space.
\end{proof}

Furthermore, since $\left\Vert \cdot \right\Vert _{G}$ and $\left\Vert \cdot
\right\Vert _{\operatorname{op}}$ are equivalent norms on $B_{\operatorname{as}}(X,X^{\left(
n-1\right) })$, then by Theorem \ref{n-dual-space-n-norm} and Theorem \ref%
{banach-space-Bss}, we get the following theorem.

\begin{theorem}
\label{banach-space-n-dual-space-Gahler}Let $X$ be a real normed space of
dimension $d\geq n$ which satisfies property (G). Then the $n$-dual space of
$(X,\left\Vert \cdot ,\cdots ,\cdot \right\Vert _{G})$ is a Banach space.
\end{theorem}

\bibliographystyle{amsplain}

\end{document}